\documentclass[11pt,reqno]{amsart}
\usepackage[T1]{fontenc}
\usepackage{amsthm}
\usepackage{amsmath}
\usepackage{stmaryrd}
\usepackage{cite}
\usepackage{amssymb}
\usepackage{xcolor}
\usepackage{enumitem}
\usepackage{mathrsfs}
\usepackage[margin=1.2in]{geometry}
\newcommand{\vertiii}[1]{{\left\vert\kern-0.25ex\left\vert\kern-0.25ex\left\vert #1 
    \right\vert\kern-0.25ex\right\vert\kern-0.25ex\right\vert}}

\usepackage{hyperref}
\hypersetup{
    colorlinks,
	linkcolor={black!30!blue},
	citecolor={red},
	urlcolor={black!30!green}
}

\usepackage{todonotes}

\newcommand{\bbC}{{\mathbb{C}}}
\newcommand{\bbD}{{\mathbb{D}}}
\newcommand{\bbN}{{\mathbb{N}}}
\newcommand{\bbQ}{{\mathbb{Q}}}
\newcommand{\bbR}{{\mathbb{R}}}

\newcommand{\bbT}{{\mathbb{T}}}
\newcommand{\bbZ}{{\mathbb{Z}}}

\newcommand{\calE}{{\mathcal{E}}}

\newcommand{\calH}{{\mathcal{H}}}
\newcommand{\calL}{{\mathcal{L}}}
\newcommand{\calM}{{\mathcal{M}}}
\newcommand{\calP}{{\mathcal{P}}}

\newcommand{\rmd}{{\mathrm{d}}}

\DeclareMathOperator{\spn}{span}

\DeclareMathOperator{\real}{Re}
\renewcommand{\Re}{\real}

\newcommand{\iop}{{\mathbf{i}}}

\newcommand{\UAMO}{{\mathrm{UAMO}}}
\newcommand{\Szego}{\operatorname{Sz}}
\DeclareMathOperator{\monodromy}{mon}
\DeclareMathOperator{\discrim}{discr}
\DeclareMathOperator{\trace}{trace}
\newcommand{\ac}{\mathrm{ac}}
\newcommand{\complexCircle}{{\partial \bbD}}

\newcommand{\idty}{{\mathbf{1}}}

\newcommand{\doubleBracket}[1]{\llbracket #1 \rrbracket}

\newtheorem{theorem}{Theorem}[section]
\newtheorem{prop}[theorem]{Proposition}
\newtheorem{lemma}[theorem]{Lemma}

\theoremstyle{definition}

\numberwithin{equation}{section}

\title[(Almost) Ballistic  Motion for Unitary Operators]{Absence of Ballistic Motion and
\\ Presence of Almost-Ballistic Motion\\
for Unitary Operators with Pure Point Spectrum}

\author[C.\ Cedzich]{Christopher Cedzich}
\address{[C.\ Cedzich] Fakult\"at f\"ur Mathematik und Informatik, FernUniversit\"at in Hagen, Universit\"atsstr. 1, 58097 Hagen, Germany}
\email{\href{mailto:christopher.cedzich@fernuni-hagen.de}{christopher.cedzich@fernuni-hagen.de}}

\author[J.\ Fillman]{Jake Fillman}
\address{[J.\ Fillman] Department of Mathematics, Texas A\&M University, College Station, TX 77843,
USA}
\email{\href{mailto:fillman@tamu.edu}{fillman@tamu.edu}}

\author[L.\ Vel\'azquez]{Luis Vel\'azquez}
\address{[L.\ Vel\'azquez] Departamento de Matem\'atica Aplicada \& IUMA, Universidad de Zaragoza, Mar\'ia de Luna 3, 50018 Zaragoza, Spain.}
\email{\href{mailto:velazque@unizar.es}{velazque@unizar.es}}

\date{}

\allowdisplaybreaks

\begin{document}

\begin{abstract}
     We adapt two results of Simon and collaborators to the setting of discrete-time unitary dynamics.
     We show that pure point spectrum precludes ballistic motion, and exhibit a family of examples showing that this is sharp within the class of extended Cantero--Moral-Vel\'{a}zquez (CMV) matrices: that is, there exist extended CMV matrices exhibiting pure point spectrum together with quantum dynamics as close to ballistic motion as one desires.
\end{abstract}
\dedicatory{Dedicated to Barry Simon on the occasion of his 80th birthday.}

\maketitle

\section{Introduction}
\subsection{Overview}
This paper concerns the quantum dynamical characteristics of unitary operators on the Hilbert space $\ell^2(\bbZ)$ with finite-range interactions.
Such operators are called ``quantum walks'' in the literature and, after suitable regrouping of the Hilbert space into local cells \cite{shiftcoin}, they describe the motion of a single particle with an internal degree of freedom in discrete time \cite{ambainisOnedimensionalQuantumWalks2001, meyerQuantumCellularAutomata1996}.
The simplest operators within this class are the extended
Cantero--Moral--Vel\'{a}zquez (CMV) matrices \cite{canteroFivediagonalMatricesZeros2003}, which play an important role in the theory of orthogonal polynomials on the unit circle (OPUC); see \cite{Simon2005OPUC1, Simon2005OPUC2} and references therein.

From the point of view of quantum simulation \cite{feynmanSimulatingPhysicsComputers1982}, one is naturally interested in the dynamical properties of discrete time evolutions, that is, the behavior of $U^t \psi$ as $t \to \infty$ for a given initial state $\psi$ and a unitary operator $U$ that satisfies a suitable locality condition.
Known results in this direction include ballistic transport (that is, linear scaling of the position observable) in the periodic setting \cite{grimmett_weak_2004,AVWW2011JMP,DamFilOng2016JMPA}, diffusive transport in the presence of decoherent noise \cite{AVWW2011JMP}, anomalous transport in certain quasicrystals \cite{DamFilOng2016JMPA}, and dynamical localization in the setting of i.i.d.\ randomness for various model classes \cite{ASW2011JMP,boumaza_dynamical_2025,hamza_dynamical_2009,joye_dynamical_2010,asch_dynamical_2012}. 
A particularly surprising type of transport is the hierarchical motion found for quantum walks in well-approximable irrational electric fields \cite{CRWAGW2013PRL}: better and better revivals of the initial state alternate with farther and farther excursions.
Moreover, one can examine conditions under which the pair $(U,\psi)$ is ``recurrent'', that is, whether $U^t\psi$ returns its initial state $\psi$ \cite{bourgain_quantum_2014,grunbaum_recurrence_2013,grunbaum_quantum_2020}.

Often, the time-dependent problem is first studied through the lens of the (time-independent) spectral problem for the unitary operator implementing the dynamics. 
More specifically, the RAGE theorem\footnote{The name here was coined by Reed and Simon \cite{reedMethodsModernMathematical1979} to celebrate important contributions of Ruelle \cite{ruelle1969remark}, Amrein and Georgescu \cite{amrein1973characterization} and Enss \cite{enss1978asymptotic}} supplies concrete (albeit coarse) connections between the spectral decomposition of an operator and the associated (unitary) dynamics: broadly speaking, the part of the wavepacket corresponding to the pure point spectrum is localized and the part corresponding to the continuous spectrum is delocalized.
More precisely, if the spectral measure of the initial state $\psi$ is pure point, then for any $\varepsilon>0$, there is a ($\varepsilon$-dependent) region of space that the wavepacket does not escape with probability $1-\varepsilon$.
On the other hand, if the spectral measure of the initial state is continuous, then the associated wave packet will leave any fixed compact set (in a time-averaged sense) as $t \to \infty$.
A more precise formulation in the discrete-time setting can be found, e.g., in \cite{enssBoundStatesPropagating1983, FillmanOng2017JFA}.
The first main result is a refinement of this discussion for banded unitary operators in $\ell^2(\bbZ)$: point spectrum precludes ballistic transport (the fastest possible transport in the setting under consideration); the analogous result for Schr\"odinger operators is due to Simon and can be found in \cite{simonAbsenceBallisticMotion1990}.
Further results and discussion about ballistic motion can be found in the survey \cite{DMYballistic}.

Let us be more precise. We say that a unitary operator $U:\ell^2(\bbZ) \to \ell^2(\bbZ)$ is \emph{banded} if for some $R>0$ one has
\begin{equation}
    |\langle \delta_n, U \delta_m \rangle | = 0 \text{ whenever } |n-m| >R.
\end{equation}
Also, we define the position operator $X:D(X) \to \ell^2(\bbZ)$ by
\begin{align}
    [X\psi ](n) 
    & = n \psi(n),\qquad %\\[2mm]
    D(X) 
    %& 
    = \left\{ \psi \in \bbC^\bbZ : \sum_{n \in \bbZ} |n\psi(n)|^2 < \infty \right\}.
\end{align}

\begin{theorem}\label{t.ppnoballistic}
Suppose $U$ is a banded unitary operator that has only pure point spectrum and for which all eigenvectors belong to $D(X)$.\footnote{In particular, this assumption is met by any banded $U$ having exponentially decaying eigenvectors.} 
Then, for every $\psi \in D(X)$, we have
\begin{equation} \label{eq:operatorNoBallistic}
\lim_{t\to\infty}\frac1{t^2}\|XU^t\psi\|^2=0.    
\end{equation}
\end{theorem}

Let us make a few comments about Theorem~\ref{t.ppnoballistic}. First, this is not a trivial consequence of the general statement from the RAGE theorem to the effect that the wave function remains confined only with probability $1-\varepsilon$.  That is, it is nevertheless possible that small portions of the wave packet do escape to infinity; in fact, there exist examples (such as the random dimer model) exhibiting a coexistence of pure point spectrum and almost-ballistic transport \cite{DamFilOng2016JMPA, deBievreGerminet2000JSP, JSS2003CMP}. 
Second, there is an additional assumption that is not present compared to the case of Schr\"odinger operators: namely, we explicitly need to assume that the eigenvectors belong to $D(X)$.
The reason for the assumption is relatively plain from the proof. When $H$ is a discrete Schr\"odinger operator,\footnote{The argument described here also generalizes to Jacobi matrices with suitable modifications.} the momentum operator $P = \iop [X,H]:=\iop(XH-HX)$ amounts to $\iop (S-S^{*})$ where $S$ is the shift $\delta_n \mapsto \delta_{n+1}$ on $\ell^2(\bbZ)$. Then, a simple computation shows
\begin{align*}
    \langle \psi, P \psi \rangle = \iop \langle \psi, S \psi - S^{*} \psi \rangle = 0 
\end{align*}
for any real-valued $\psi$. Since the eigenvectors of $H$ can be chosen to be real-valued, this shows that $\langle \psi, P \psi \rangle = 0$ for all eigenvectors $\psi$ of $H$ directly.
However, in the present situation, the ``momentum'' operator $P=[X,U]$ is in general more complicated, and it is easy to see that, in general, one does not have $\langle \psi, P\psi \rangle = 0$ for all real-valued $\psi$ (and at any rate, eigenfunctions of a general banded unitary are not, in general, amenable to being realified in the same way as eigenfunctions of a Schr\"odinger operator).
We instead need to use a different (formal) computation to directly show that $\langle \psi, P\psi \rangle = 0$ for eigenvectors $\psi$ of $U$, but the formal computation requires $\psi \in D(X)$.
It would be interesting to clarify whether this assumption is needed or if a different argument can get rid of it, perhaps in suitable special cases.
Finally, the result still holds after regrouping the Hilbert space $\ell^2(\bbZ)$ into (finite-dimensional) local cells as is done for quantum walks \cite{shiftcoin}. This requires a modification of the definition of the position operator, but it is straightforward to see that the proof still applies after the necessary changes.

\subsection{A Pathological Example Among ECMV Matrices}
\label{subsec:ECMVbasics}
Among the banded unitary matrices, one of the most important classes of examples is that of extended Cantero--Moral--Vel\'azquez (ECMV) matrices.
Given coefficients $\alpha_n \in \bbD := \{z \in \bbC : |z|  < 1\}$ for $n \in \bbZ$, the associated ECMV matrix is defined by 
$\calE=\calL\calM$, where 
\begin{equation} \label{eq:LMdecomp}
\calL=\bigoplus_{n\in\bbZ}\Theta(\alpha_{2n}), \qquad 
\calM=\bigoplus_{n\in\bbZ}\Theta(\alpha_{2n+1})
\end{equation}
are specified by
\begin{equation} \label{eq:theta_mat}	
\Theta(\alpha)
= \begin{bmatrix}\overline{\alpha} & \rho\\
	{\rho}&-\alpha\end{bmatrix}, 
    \quad \rho := \sqrt{1-|\alpha|^2},
\end{equation}
with the convention that $\Theta(\alpha_j)$ acts on the subspace $\ell^2(\{j,j+1\})$.
On one hand, this setting is natural from the point of view of orthogonal polynomials on the unit circle and harmonic analysis on the unit disk.
On the other hand  it is quite natural to consider complexified $\rho$'s in some scenarios, which directly leads to operators that are (equivalent to) split-step quantum walks on the integer lattice.
On account of the CGMV connection and its antecedents \cite{CFLOZ2024IMRN, CGMV2010CPAM, CGMV2012QIP}\footnote{Note that the Arnold principle applies here as recited in \cite{simonCMVMatricesFive2007}: this connection was noted at least a decade earlier in \cite{AlainUnitaryBandMats} (without explicit reference to ``quantum walks'') and then rediscovered several times.} the two settings are equivalent up to a simple gauge transformation.
We emphasize that the freedom to adjust the phases of the $\rho$'s is not merely an obsolete artifact but is often convenient if not necessary to identify hidden symmetries and facilitate proofs, see for example \cite{CFLOZ2024IMRN, CF2024JST}.

From a suitable member of this class, we will see that Theorem \ref{t.ppnoballistic} cannot be improved in the sense that pure point spectrum can coexist with almost-ballistic motion.
    More precisely, we have the following:
\begin{theorem} \label{t:quasiballistic}
     For any increasing function $f:(0,\infty) \to (0,\infty)$ with $f(t) \to \infty$ as $t \to \infty$ there exists a GECMV matrix $\calE$ having pure point spectrum, exponentially decaying eigenvectors, and satisfying an almost-ballistic motion bound of the form
    \begin{equation}\label{eq:almostballistic}
        \limsup_{t\to\infty} \frac{f(t)}{t^2}\|X\calE^t\delta_0\|^2 = \infty.
    \end{equation}
\end{theorem}

The examples for Theorem~\ref{t:quasiballistic} will be constructed from suitable finite-rank perturbations of the unitary almost-Mathieu operator \cite{CFO2023CMP}, in the spirit of the constructions in \cite{dRJLS1996JAM, Last1996JFA}.
\bigskip

Let us conclude our introductory comments with a general philosophical comment about the choice of setting.
Often a unitary operator $U$ can be written as $U = \exp(-\iop H)$ for a suitable self-adjoint operator $H$, so one may naturally wonder why one does not simply work with the generator $H$ and extant results for the unitary group $\exp(-\iop tH)$, which could then simply be specialized to $t \in \bbZ$. 
However, there are (at least) three problems with this perspective. 
First, not every banded unitary enjoys a sufficiently local self-adjoint generator, which can be seen from the example of the shift $[U\psi](n) = \psi(n+1)$.
Second, and more importantly, the unitary operator $U$ under consideration will typically have some structure that one wishes to leverage, and such structure may be completely obscured or destroyed in the passage to a self-adjoint generator.
Finally, there has been recent success at implementing such discrete time quantum simulators using different platforms such as neutral atoms in optical lattices \cite{karskiQuantumWalkPosition2009} or time-multiplexed setups using coherent light \cite{schreiberPhotonsWalkingLine2010} or single photons \cite{xueObservationQuasiperiodicDynamics2014}, with even subtle dynamical and spectral features such as a metal-insulator transition being visible after a few time steps \cite{ElektricExp,ExpUAMO}.

\subsection*{Acknowledgements} J.F.\ was supported in part by National Science Foundation grant DMS-2513006 and Simons Foundation Grant MPS-TSM-00013720. L.V.\ was supported in part by the project PID2021-124472NBI00 funded by MICIU/AEI/10.13039/501100011033 and by `ERDF A way of making Europe'.

\section{Absence of Ballistic Motion} \label{sec:proofnoball}

Before we give the proof of Theorem~\ref{t.ppnoballistic}, let us make a few preliminary remarks. 
In general, if $B$ is an operator, we will write 
%\begin{equation} \label{eq:basSpr:heisEvoDef} 
    $B_U(t) = U^{-t}  B U^t$, 
    %\quad 
    $t \in \bbZ$,
%\end{equation} 
for the Heisenberg evolution of $B$ with respect to $U$. 
Whenever confusion is unlikely, we will suppress $U$ and simply write $B(t)$. 
For $B = X$, the position operator, we note the following elementary fact, which can be proved by direct computations:
\begin{prop}
   If $U$ is a banded unitary, then $D(X)$ is $U$-invariant and $X(t)$ is self-adjoint on $D(X(t)) = D(X)$ for all $t \in \bbZ$. 
\end{prop}

We then observe
\begin{align}
\nonumber
X(t+1) - X(t)
%& = U^{-t-1} X U^{t+1}
%- U^{-t} X U^t \\
%\nonumber
& = U^{-1} U^{-t} (XU - U X)U^{t} %\\
%& 
= U^{-1} P(t)
\end{align}
where $P$ is the discrete momentum operator
\begin{equation} \label{eq:basSpr:Pdef}
P :=[X,U] = X U - U X.
\end{equation}
Since $X$ is unbounded, $P$ is initially only defined on $D(X)$; however, by direct computations, it is readily seen that $P$ is bounded on $D(X)$ and hence (since $D(X)$ is dense in $\ell^2(\bbZ)$) extends uniquely to a bounded linear operator on all of $\ell^2(\bbZ)$.
Thus, we have
\begin{equation}\label{e.simonppnoballistic1b}
X(t) - X 
= U^{-1} \sum_{s=0}^{t-1} P(s)
\end{equation}
on $D(X)$.

\begin{lemma} \label{lem:eigConv}
    Let $\phi$ and $\psi$ be eigenvectors of a banded unitary $U$ satisfying the assumptions of Theorem \ref{t.ppnoballistic}.
    Then
    \begin{equation} \label{eq:SimonClaimGoal}
        %\lim_{t \to \infty} \frac{1}{t^2} \sum_{s,\widetilde{s}=0}^{t-1} \langle \phi , U^{-1} P(s) U^{-1} P(\widetilde s)  \psi \rangle
        %= 
        \lim_{t \to \infty} \frac{1}{t^2} \sum_{s,\widetilde{s}=0}^{t-1} \langle P(s)\phi , P(\widetilde s)  \psi \rangle
        = 0.
    \end{equation}
\end{lemma}

\begin{proof}
Let $\{\varphi_1,\ldots \}$ denote a complete orthonormal set of eigenvectors of $U$ with corresponding eigenvalues $\{z_1, \ldots \}$. Let $w$ and $z$ be the eigenvalues corresponding to $\phi$ and $\psi$, respectively.
%Utilizing \eqref{eq:Ps_conj} and e
Expanding in the basis $\{\varphi_j\}$ yields
\begin{align*}
\left\langle P(s) \phi , P(\widetilde s) \psi \right\rangle 
& = w^{-s}z^{\widetilde{s}} \left\langle U^{-s} P \phi ,U^{  -\widetilde{s}} P  \psi
\right\rangle \\
& = w^{-s } z^{ \widetilde{s}} 
\left\langle P \phi , \sum_{j = 1}^\infty \left\langle \varphi_j, U^{s- \widetilde{s}} P  \psi \right\rangle \varphi_j \right\rangle \\
& = w^{-s } z^{ \widetilde{s}} \sum_{j = 1}^\infty z_j^{s-\widetilde s}\left\langle \varphi_j, P  \psi \right\rangle\left\langle P \phi ,   \varphi_j \right\rangle.
\end{align*}

By the Cauchy--Schwarz inequality, one has
\begin{align*}
\sum_{j=1}^\infty |\langle \varphi_j,P\psi\rangle | |\langle P \phi, \varphi_j\rangle|
& \leq \left[\sum_{j=1}^\infty|\langle \varphi_j,P\psi\rangle |^2 \right]^{1/2} \left[ \sum_{j=1}^\infty|\langle P \phi, \varphi_j\rangle|^2\right]^{1/2} \\ & = \|P\psi\| \|P\phi\| \\ &< \infty.
\end{align*}
Thus, by dominated convergence, the left-hand side of \eqref{eq:SimonClaimGoal} becomes
\begin{align*}
\lim_{t \to \infty} & \frac{1}{t^2} \sum_{s,\widetilde{s}=0}^{t-1}
\left\langle P(s)\phi , P(\widetilde s) \psi \right\rangle  %\\
%& 
= \sum_{j = 1}^\infty \lim_{t \to \infty} \frac{1}{t^2} \sum_{s,\widetilde{s}=0}^{t-1} w^{-s}z_j^{s-\widetilde s} z^{\widetilde s} \left\langle \varphi_j, P  \psi \right\rangle \left\langle P \phi , \varphi_j \right\rangle.
\end{align*}
When $z \neq w$, the quantity on the right hand side goes to zero by a direct computation.
In the case $w = z$, only the terms with $z_j = z = w$ survive, so the quantity in question reduces to
\begin{align*}
%& \sum_{j = 1}^\infty \lim_{t \to \infty} \frac{1}{t^2} \sum_{s,\widetilde{s}=0}^{t-1} w^{-s}z_j^{s-\widetilde s} z^{\widetilde s} \left\langle \varphi_j, P  \psi \right\rangle \left\langle P \phi , \varphi_j \right\rangle %\\
\lim_{t \to \infty} \frac{1}{t^2} \sum_{s,\widetilde{s}=0}^{t-1}
\left\langle P(s)\phi , P(\widetilde s) \psi \right\rangle 
%& \qquad 
= \sum_{j : z_j = z} \langle \varphi_j, P\psi \rangle \langle P\phi,\varphi_j\rangle.
\end{align*}
The proof will be complete once we show that 
\(
\langle \varphi_j, P\psi \rangle
 = 0
=  \langle P \phi , \varphi_j \rangle 
\)
for all $j$ such that $z_j = z=w$.
Since $\psi$ and $\varphi_j$ belong to $D(X)$, a direct computation produces
\begin{align*}
    \langle \varphi_j, P \psi \rangle
%     & = \langle \varphi_j, (XU - U X)\psi \rangle \\
%     & = \langle \varphi_j, X U \psi \rangle  - \langle X U^{-1}\varphi_j, \psi \rangle \\
      = \langle \varphi_j, X z \psi \rangle  - \langle X z^{-1}\varphi_j, \psi \rangle 
%     & = \langle \varphi_j, X z \psi \rangle  - \langle \varphi_j, Xz\psi \rangle \\ 
      = 0,
\end{align*}
and a similar computation shows that $\langle P\phi, \varphi_j \rangle = 0$.
\end{proof}

\begin{proof}[Proof of Theorem~\ref{t.ppnoballistic}]
Let $U$ be an operator satisfying the assumptions of the theorem, and let $\{\varphi_1,\ldots \}$ denote a complete orthonormal set of eigenvectors of $U$ with corresponding eigenvalues $\{z_1, \ldots \}$.

From \eqref{e.simonppnoballistic1b} we obtain for $\psi \in D(X)$:
\begin{align*}
\| X U^t \psi \|^2 
%& = \langle X U^t \psi , X U^t \psi \rangle \\
%& = \langle \psi , U^{-t} X^2 U^t \psi \rangle \\
%& = \langle \psi , U^{-t} X U^{t} U^{-t} X U^{t} \psi \rangle \\
%& = \left\langle \psi , \left( X+  U^{-1} \sum_{s=0}^{t-1} P(s) \right) \; \left(  X+U^{-1} \sum_{\widetilde{s}=0}^{t-1} P(\widetilde{s})\right) \psi \right\rangle.
& = \| U^{-t}X U^t \psi \|^2  =
%\|X(t)\psi\|^2  =
\left\|\left( X+  U^{-1} \sum_{s=0}^{t-1} P(s) \right)\psi \right\|^2
=\sum_{s,\widetilde{s}=0}^{t-1}  \left\langle P(s)\psi , P(\widetilde{s}) \psi \right\rangle + O(t),
\end{align*}
where in the last step we utilized that $P$ is bounded.
%Since $P$ is bounded, using \eqref{eq:Ps_conj} the final expression can be written as
%\begin{equation}
   %\sum_{s,\widetilde{s}=0}^{t-1}  \left\langle \psi ,   U^{-1}  P(s)   U^{-1} P(\widetilde{s}) \psi \right\rangle + O(t)=
%   \sum_{s,\widetilde{s}=0}^{t-1}  \left\langle P(s)\psi , P(\widetilde{s}) \psi \right\rangle + O(t).
%\end{equation}
Consequently, 
\begin{equation}\label{e.simonppnoballistic2}
%\lim_{t \to \infty} \frac{1}{t^2} \| X  U^t \psi \|^2  = \lim_{t \to \infty} \frac{1}{t^2} \sum_{s,\widetilde{s}=0}^{t-1} \langle \psi , U^{-1} P(s) U^{-1} P(\widetilde s)  \psi \rangle.
\lim_{t \to \infty} \frac{1}{t^2} \| X  U^t \psi \|^2  = \lim_{t \to \infty} \frac{1}{t^2} \sum_{s,\widetilde{s}=0}^{t-1}  \left\langle P(s)\psi , P(\widetilde{s}) \psi \right\rangle.
\end{equation}

Now, let us assume that $\psi \in D(X)$ is a linear combination of (finitely many) members of $\{\varphi_j\}$, say,
$
\psi = \sum_{j = 1}^N c_j \varphi_j.
$
We have
\begin{align*}
\frac{1}{t^2} \langle P(s)\psi ,  P(\widetilde s) \psi \rangle
& = \frac{1}{t^2}  \left\langle \sum_{k = 1}^N c_k P(s)\varphi_k , P(\widetilde s) \sum_{j = 1}^N c_{j} \varphi_{j} \right\rangle \\
& = \frac{1}{t^2} \sum_{k = 1}^N \sum_{j = 1}^N \overline{c_k} c_{j} \left\langle P(s)\varphi_k , P(\widetilde s) \varphi_{j} \right\rangle,
\end{align*}
which converges to zero by the Lemma \ref{lem:eigConv}.
Since the vectors $\{\varphi_j\}$ comprise an orthonormal basis, the proof is finished.
\end{proof}

\section{A Pathological Example} \label{sec:pathological}

In this section, we prove Theorem \ref{t:quasiballistic} by constructing an explicit family of extended CMV matrices that have pure point spectrum with exponentially decaying eigenfunctions while displaying almost-ballistic transport in the sense of \eqref{eq:almostballistic}. These  matrices build on the unitary almost-Mathieu operator (UAMO) introduced in \cite{FilOngZha2017CMP, CFO2023CMP}: we consider $\calE_{\lambda_1, \lambda_2, \Phi,\theta}^\UAMO$ given by the Verblunsky coefficients
\begin{equation}
\begin{split}
&\alpha_{2n-1}^\UAMO  = \lambda_2 \sin 2\pi(n\Phi + \theta) , \hspace{18mm}
\alpha_{2n}^\UAMO = \lambda_1',
\end{split}
\end{equation}
where $\lambda' := \sqrt{1-\lambda^2}$.
This fits into the context of ergodic CMV matrices with the base dynamics $(\theta,j) \mapsto (\theta + \Phi/2,j+1)$ on the compact abelian group $\bbT \times \bbZ_2$ and sampling function
\begin{equation}
    f(\theta,j)
    = \begin{cases}
        \lambda_2 \sin 2\pi \theta & j =0 \\
        \lambda_1' & j=1.
    \end{cases}
\end{equation}

In the supercritical regime, $\lambda_1<\lambda_2$, it is known that $\calE_{\lambda_1,\lambda_2,\Phi,\theta}^\UAMO$ exhibits Anderson localization for a full measure set of $\Phi$'s \cite{CFO2023CMP, Yang2024Nonlin}, that is, for a.e.\ $\Phi,\theta$, $\calE_{\lambda_1,\lambda_2,\Phi,\theta}^\UAMO$ has a complete set of exponentially decaying eigenvectors.

Fix such $\lambda_1<\lambda_2$ and consider for $\beta_0,\beta_1\in \bbD$, the variations $\calE_{\Phi,\theta,\beta_0,\beta_1}$ of $\calE_{\lambda_1,\lambda_2,\Phi,\theta}^\UAMO$ given by
\begin{equation}\label{eq:Verblunskys_UAMO_mod}
\alpha_n = \begin{cases}
    \beta_n & n=0,1 \\
    \alpha_n^\UAMO & \text{otherwise.}
\end{cases}
\end{equation}

\begin{theorem} \label{t:pathologicalExample}
    For any increasing function $f:(0,\infty) \to (0, \infty)$ with $f(t) \to \infty$ as $t \to \infty$ and any $\lambda_1,\lambda_2\in[0,1]$ with $\lambda_1<\lambda_2$ there exist $\Phi$, $\theta$, $\beta_0$, and $\beta_1$ such that  $\calE = \calE_{\Phi,\theta,\beta_0,\beta_1}$ exhibits the coexistence of Anderson localization and almost-ballistic motion of the form
    \begin{equation} \label{eq:quasiBall}
        \limsup_{t\to\infty} \frac{f(t)}{t^2} 
        \|X\calE^t \delta_0 \|^2
        = \infty.
    \end{equation}
\end{theorem}

Throughout the discussion, let $f:(0,\infty) \to (0,\infty)$ denote an increasing function satisfying $f(t) \to \infty$ as $t \to \infty$. The dynamical statement from \eqref{eq:quasiBall} is sometimes called \emph{quasiballistic motion}; see \cite{JitZha2022JEMS} for results that can be obtained from quantitative repetition properties for self-adjoint operators.

\subsection{A Discriminant Estimate}

We will need a simple estimate on the discriminant of a suitable periodic problem, which naturally arises when the frequency $\Phi$ is rational.
Given $\alpha  \in \bbD$ and $z \in \bbC$, the associated \emph{Szeg\H{o} matrix} is given by
\begin{equation}
\Szego(\alpha,z) = \frac{1}{\rho} \begin{bmatrix} z & - \overline{\alpha} \\ -\alpha z & 1 \end{bmatrix}, 
\quad \rho := \sqrt{1-|\alpha|^2}.
\end{equation}
Using this, we define
\begin{equation}\label{eq:Szegö_prod}
\Szego_{[n,m]}(z) \equiv \Szego_{[n,m]}(z;\calE) 
:= \Szego(\alpha_n,z) \Szego(\alpha_{n-1},z) \cdots \Szego(\alpha_m,z)
\end{equation}
 If $(\alpha_n)_{n \in \bbZ}$ is periodic, that is, for some $p \in \bbN$, $\alpha_{n+p} \equiv \alpha_n$, we may study the spectral and dynamical characteristics of the associated CMV operator $\calE$ with the help of a Floquet decomposition, which we briefly recall here.
First, since we do not require minimality of the period, we freely assume that the period is $p=2q$ with $q \in \bbN$.
The \emph{monodromy} and \emph{discriminant} are given respectively by
\begin{equation}
    \monodromy(z) = z^{-q} \Szego_{[1,2q]}(z), \quad
    \discrim(z) = \trace \monodromy (z).
\end{equation}
The equation $\discrim(z) = 0$ has $2q$ distinct solutions, which we order according to argument and denote by $z_1(\pi/2),\ldots,z_{2q}(\pi/2)$. We then continuously extend $z_j(k)$ so that
\begin{equation} \label{eq:discrimininantFloquetEigs}
    \discrim z_j(k) = 2\cos k, \quad k \in [0,\pi], \ j = 1,2,\ldots,2q.
\end{equation}
Differentiating both sides of \eqref{eq:discrimininantFloquetEigs} and rearranging gives
\begin{equation} \label{eq:implicit}
    \frac{\rmd z_j}{\rmd k} 
    = - \frac{2\sin k}{\discrim' (z_j(k))}.
\end{equation}
Using the product rule for differentiation, we see that
\begin{equation}
    |\discrim'(z)| \leq C^q, \quad z \in \complexCircle,
\end{equation}
where $C = C(\alpha) > 0$ is a constant that depends only on $\max|\alpha_j|$. Putting this together with \eqref{eq:implicit}, one obtains the useful elementary lower bound
\begin{equation} \label{eq:dz/dkLowerBound}
    \left|\frac{\rmd z_j}{\rmd k} \right|
    \geq  2C^{-q} \sin k , \quad k \in [0,\pi].
\end{equation}

\subsection{Dynamical Estimates}

We begin by showing two lemmata that are independent of the model at hand: the first one provides an iterative criterion for almost-ballistic motion while the second one gives some useful bounds for general unitary operators.

\begin{lemma}\label{l.almostballisticsufficientcondition}
Let $U$ be a unitary operator on $\ell^2(\bbZ)$. If there exist $T_m \to \infty$ satisfying
\begin{equation}\label{e.almostballisticcriterion}
\frac{1}{T_m} \sum_{t=T_m}^{2T_m-1} \sum_{|n| \ge \frac{T_m}{f(T_m)^{1/5}}} |\langle \delta_n, U^t \delta_0 \rangle|^2
\ge \frac{1}{f(T_m)^{2/5}}, \qquad m \in \bbN,
\end{equation}
then \eqref{eq:quasiBall} holds.
\end{lemma}

\begin{proof}
For each $m \in \bbN$, \eqref{e.almostballisticcriterion} implies that for some integer $t_m \in [T_m, 2T_m-1]$, we have
$$
\sum_{|n| \ge \frac{T_m}{f(T_m)^{1/5}}} |\langle \delta_n, U^{t_m} \delta_0 \rangle|^2 \ge \frac{1}{f(T_m)^{2/5}}.
$$
But then
\begin{align*}
\limsup_{t\to\infty} \frac{f(t)}{t^2} 
        \|XU^t \delta_0 \|^2
        & =
\limsup_{t \to \infty} \frac{f(t)}{t^2}  \sum_{n\in \bbZ} |n|^2 |\langle \delta_n, U^t \delta_0 \rangle|^2  \\
&  \ge \limsup_{m \to \infty} \frac{f(t_m)}{t_m^2}  \sum_{n\in \bbZ} |n|^2 |\langle \delta_n, U^{t_m} \delta_0 \rangle|^2 \\[1mm]
&   \ge \limsup_{m \to \infty} \frac{f(t_m)}{(2T_m)^2} \frac{T_m^2}{f(T_m)^{2/5}} \sum_{|n| \ge \frac{T_m}{f(T_m)^{1/5}}} |\langle \delta_n, U^{ t_m} \delta_0 \rangle|^2 \\%[1mm]
%&   \ge \limsup_{m \to \infty} \frac{f(t_m)}{(2T_m)^2} \frac{T_m^2}{f(T_m)^{2/5}} \frac{1}{f(T_m)^{2/5}} \\[1mm]
&   \ge \limsup_{m \to \infty} \frac{f(T_m)^{1/5}}{4} \\[1mm]
&   = \infty,
\end{align*}
as claimed.
\end{proof}

\begin{lemma}
\label{l.almostballisticbreakdown}
Suppose  $\calH$ is a Hilbert space, $P:\calH \to \calH$ is an orthogonal projection, $U:\calH \to \calH$ is unitary,  and $\delta = \varphi + \psi$ is a unit vector with $\langle \varphi, \psi \rangle = 0$.
Then,
\begin{equation}
\label{eq:genericProjEst}
\frac{1}{T} \sum_{t=T}^{2T-1} \| (\idty - P) U^t \delta\|^2 
\ge \| \psi \|^2 - 3 \left( \frac{1}{T} \sum_{t=T}^{2T-1} \| P U^t \psi\|^2  \right)^{1/2}
\end{equation}
for every $T \in \bbN$.
\end{lemma}

\begin{proof}
Since $U$ is unitary, $P$ is an orthogonal projection, $\delta = \varphi + \psi$ is a unit vector, and $\langle \varphi, \psi \rangle  = 0$, the left-hand side of \eqref{eq:genericProjEst} can be rewritten as
\begin{align}
\nonumber
\frac{1}{T} 
\sum_{t=T}^{2T-1} \| (\idty - P) U^t \delta\|^2   
& = \frac{1}{T} \sum_{t=T}^{2T-1} 
\left( \|U^t(\varphi+\psi)\|^2 - \|P U^t (\varphi+\psi)\|^2 \right)   \\[1mm]
\nonumber
& =  \| \varphi \|^2 + \| \psi \|^2 -\frac{1}{T} \sum_{t=T}^{2T-1} 
\left( \|P U^t \varphi\|^2 
+ \|P U^t \psi\|^2 
+ 2 \real \langle P U^t \varphi , P U^t \psi  \rangle \right)   \\[1mm]
\nonumber
& \geq \|\varphi\|^2 
+ \| \psi \|^2  
- \frac{1}{T} \sum_{t=T}^{2T-1} 
\left( \|\varphi\|^2 +  \|P U^t \psi\|^2 
+ 2 \real \langle P U^t \varphi , P U^t \psi  \rangle \right)  \\
\label{eq:doublbracketEst1}
& = \|\psi\|^2 - \doubleBracket{\psi,\psi}_T - 2\real \doubleBracket{\varphi, \psi}_T,
\end{align}
where 
\begin{equation}
\label{e.almostballisticSLF}
\doubleBracket{x,y}_T := \frac{1}{T} \sum_{t=T}^{2T-1} \langle P U^t x, P U^t y \rangle  ,
\quad x,y \in \calH, \ T \in \bbN.
\end{equation}
Since $P$ is an orthogonal projection and $U$ is unitary, we note that one directly has $|\doubleBracket{x,y}_T| \leq \|x\|\|y\|$ for arbitrary $x,y$. 
Furthermore, by the Cauchy--Schwarz inequality, we get
\begin{align}\nonumber
    |\doubleBracket{x,y}_T|
    & \leq
    \frac{1}{T} \sum_{t=T}^{2T-1} \| P U^t x \| \| P U^t y \| \\
    \nonumber
    & \leq \left[ \frac{1}{T} \sum_{t=T}^{2T-1} \| P U^t x \|^2 \right]^{1/2} \left[  \frac{1}{T} \sum_{t=T}^{2T-1} \| P U^t y \|^2 \right]^{1/2} \\
    \label{eq:Qxybound}
    & =
    \doubleBracket{x,x}_T^{1/2} \doubleBracket{y,y}_T^{1/2}.
\end{align}

Using   \eqref{eq:Qxybound}  to estimate \eqref{eq:doublbracketEst1}, we note that
\begin{align*}
&\left( \doubleBracket{\psi,\psi}_T + 2 \Re \doubleBracket{\varphi,\psi}_T \right)^2 \\
& \qquad = 
\doubleBracket{\psi,\psi}_T^2 
+ 4 \doubleBracket{\psi,\psi}_T \cdot 
\Re \doubleBracket{\varphi,\psi}_T 
+ 4 \left[ \Re \doubleBracket{\varphi,\psi}_T \right]^2 \\
& \qquad \le \doubleBracket{\psi,\psi}_T + 4 \doubleBracket{\psi,\psi}_T + 4 \doubleBracket{\psi,\psi}_T \\
& \qquad = 9 \doubleBracket{\psi,\psi}_T,
\end{align*}
which in turn gives
\begin{equation}\label{e.dRJLS.LB2.2}
\| \psi \|^2 - \doubleBracket{\psi,\psi}_T  - 2 \Re \doubleBracket{\varphi,\psi}_T 
\ge \| \psi \|^2 - 3\doubleBracket{\psi,\psi}_T^{1/2},
\end{equation}
concluding the argument.
\end{proof}

The formulation of the next lemma requires some additional notation. Given a unitary operator $U$ and vectors $\varphi,\psi \in \calH$, the mixed spectral measure $\mu_{U,\varphi,\psi}$ is given by
\begin{equation}
\int f(z) \, \rmd\mu_{U,\varphi,\psi}(z)
:=
 \langle \varphi, f(U)\psi \rangle
\end{equation}
for bounded Borel measurable $f$.
When $\varphi = \psi$, we abbreviate $\mu_{U,\psi} = \mu_{U, \psi, \psi}$.
If the spectral measure $\mu_{U,\psi}$ is absolutely continuous (with respect to Lebesgue measure on $\complexCircle$, we write $F_{U,\psi}$ for its density, that is,
\begin{equation}
    \langle \psi, f(U)\psi \rangle
    = \int_{\complexCircle} f(z) F_{U,\psi}(z) \, \frac{\rmd z}{2\pi \iop z}.
\end{equation}
In this case, we define $\vertiii{\psi}_U := \|F_{U, \psi}\|_\infty^{1/2}$.
Otherwise, we put $\vertiii{\psi}_U = \infty$.
For a finite Borel measure $\mu$ on $\complexCircle$, we denote by $\widehat\mu$ its Fourier transform, given by
\begin{equation}
    \widehat{\mu}(t)
    = \int z^{-t} \, \rmd \mu(z),
    \quad t\in \bbZ.
\end{equation}
Likewise, if $F:\complexCircle \to \bbC$ is Lebesgue-integrable, we write
\begin{equation}
    \widehat F(t) = \int_{\complexCircle} z^{-t} F(z) \, \frac{\rmd z}{2\pi \iop z}..
\end{equation}

\begin{lemma} \label{lem:specFourier}
    Suppose $\calH$ is a Hilbert space, $U:\calH \to \calH$ is unitary, and $\varphi,\psi \in \calH$.
    If $\mu_\psi$ is absolutely continuous, then
    \begin{equation} \label{eq:planchGoal}
        \sum_{t \in \bbZ} |\langle \varphi, U^t \psi \rangle |^2
        \leq  \vertiii{\psi}^2 \|\varphi\|^2
    \end{equation}
\end{lemma}

\begin{proof}
By the Cauchy--Schwarz inequality, we know that $\mu_{U,\varphi,\psi}$ is absolutely continuous with respect to $\mu_{U,\psi}$, hence with respect to Lebesgue measure, so its density $F_{\varphi,\psi}$ with respect to Lebesgue measure is well-defined.
By Parseval's formula, the left-hand side of \eqref{eq:planchGoal} is equivalent to
\begin{equation} \label{eq:dynamicFT}
    \sum_t \left| \widehat{\mu}_{U,\varphi,\psi}(t) \right|^2
    = \| \widehat\mu_{U,\varphi,\psi}\|^2_{\ell^2(\bbZ)}
    = \|F_{\varphi,\psi}\|^2_{L^2(\complexCircle)}.
\end{equation}

Let $\calP_\ac$ denote the orthogonal projection to the absolutely continuous subspace of $U$.
Writing $\varphi_{\rm ac} = \mathcal{P}_{\rm ac}\varphi$, we observe that
\begin{equation}
    |F_{\varphi,\psi}| \leq F_\psi^{1/2}F_{\varphi_\ac}^{1/2}
\end{equation}
a.e.\ by the Cauchy--Schwarz inequality. Combining this with \eqref{eq:dynamicFT} gives
\begin{equation}
    \sum_t \left| \widehat{\mu}_{U,\varphi,\psi}(t) \right|^2
    \leq \int_\complexCircle |F_\psi||F_{\varphi_{\ac}}| 
    \leq \vertiii{\psi}^2\|\varphi_\ac\|^2
    \leq \vertiii{\psi}^2\|\varphi\|^2,
\end{equation}
as promised.
\end{proof}

This concludes the generalities. Now we fix $\lambda_1<\lambda_2$ and consider for $\beta_0,\beta_1$, the variations $\calE_{\Phi,\theta,\beta_0,\beta_1}$ of the unitary almost-Mathieu operator $\calE_{\lambda_1,\lambda_2,\Phi,\theta}^\UAMO$ given in \eqref{eq:Verblunskys_UAMO_mod}.

\begin{lemma} \label{eq:qp:quasiball:dynDuhaml}
For all $\Phi_1,\Phi_2 \in \bbT$ and $t \in \bbR$, we have
$$
\sup_{\theta \in \bbT, \, \beta_j \in \bbD} \left\| \calE_{\Phi_1,\theta,\beta_0,\beta_1}^t \delta_0 - \calE_{\Phi_2,\theta,\beta_0,\beta_1}^t \delta_0 \right\| 
\le 2 \pi\lambda_2|\Phi_1  - \Phi_2| t^2.
$$
\end{lemma}
Note that the right hand side is trivially bounded from above by $2$, so the statement of the lemma is only nontrivial for times $t\lesssim |\Phi_1  - \Phi_2|^{-1/2}$.
\begin{proof}[Proof of Lemma~\ref{eq:qp:quasiball:dynDuhaml}]
Consider $\calE_j =\calE_{\Phi_j,\theta,\beta_0,\beta_1}$.
%One has the following simple telescoping sum:
%\begin{align}\label{e.duhamleapp1}
%\calE_1^t - \calE_2^t  
%= \sum_{s=0}^{t-1} \calE_2^{t-s-1}\calE_1^{s+1} - \calE_2^{t-s}\calE_1^{s}
%= \sum_{s=0}^{t-1} \calE_2^{t-s-1}(\calE_1 - \calE_2 ) \calE_1^{s} .
%\end{align}
Writing $\calE_j = \calL_j \calM_j$ for the $\calL\calM$-decomposition of $\calE_j$ as in \eqref{eq:LMdecomp}, we note that $\calL_1 = \calL_2$ (and both are unitary), so we obtain the following elementary bound:
\begin{align*}
\| (\calE_1 - \calE_2) \psi\|
=  \| (\calM_1 - \calM_2) \psi\|
\le 2\pi\lambda_2 | \Phi_1 - \Phi_2| \| X \psi \|.    
\end{align*}

Combining this estimate with a telescoping sum argument, we get
\begin{equation}\label{e.duhamleapp2}
\left\| \calE_1^t \delta_0 - \calE_2^t \delta_0 \right\| 
\le 2\pi\lambda_2 |\Phi_1 - \Phi_2| 
\sum_{s=0}^{t-1} \left\| X \calE_2^s \delta_0 \right\|,
\end{equation}
uniformly in $\theta,\beta_0, \beta_1$.

Using $X \delta_0 = 0$ and $\| P(s) \| \le 2$, we can further estimate the right-hand side of \eqref{e.duhamleapp2} further and obtain
\begin{equation}\label{e.duhamleapp3}
\left\| \calE_1^{t} \delta_0 - \calE_2^{t} \delta_0 \right\| 
\le 2\pi\lambda_2  |\Phi_1 - \Phi_2| \sum_{s=0}^{t-1}2 s 
=2 \pi\lambda_2  |\Phi_1 - \Phi_2| t(t-1),
\end{equation}
from which the claim follows.
\end{proof}

\begin{lemma}\label{l.almostballisticppdecomposition}
Given $\Phi \in \bbQ$, there are constants $C_1 = C_1(\Phi)$, $C_2 = C_2(\Phi) \in (0,\infty)$ such that for any $\theta \in \bbT$ and any $\beta_0,\beta_1 \in \bbD$, there exists a decomposition
\begin{equation}
\delta_0 = \varphi_{\theta, \beta_0, \beta_1} + \psi_{\theta, \beta_0, \beta_1}
\end{equation}
where $\varphi = \varphi_{\theta, \beta_0, \beta_1}$ and $\psi = \psi_{\theta, \beta_0, \beta_1}$ have the following properties:
\begin{align}
\label{eq:qp:quasiball:phipsiconstr1}
\langle \varphi,\psi \rangle & = 0 \\
\label{eq:qp:quasiball:phipsiconstr2}
\|\psi\| & \geq C_1 \\
\label{eq:qp:quasiball:phipsiconstr3}
\vertiii{\psi}_{\calE_{\Phi,\theta,\beta_0,\beta_1}} & \leq C_2.
\end{align}
\end{lemma}

\begin{proof}
Let $\Phi \in \bbQ$ be given. First,   consider the unperturbed operator $
\calE_{\Phi,\theta} := \calE_{\Phi,\theta,\alpha_0, \alpha_1}$. 
Writing $\Phi = p/q$ in lowest terms, we see that $\calE_{\Phi,\theta}$ is periodic of period $2q$. For $
\theta \in \bbT$ and $k \in [0,2\pi]$, denote by $z_{1,\theta}(k), \ldots, z_{2q,\theta}(k)$ the Floquet eigenvalues sorted by their argument, and define
\begin{equation}
I_{\theta} = z_{1,\theta}( [\pi/3, 2\pi/3]).
\end{equation} 
Combining this with the estimate \eqref{eq:dz/dkLowerBound}, we obtain 
\begin{equation} \label{eq:qp:quasiball:Iomegalengthbd}
\ell(\Phi) := \inf_{\theta \in \bbT} |I_\theta| > 0.
\end{equation}
 We claim that the desired $\psi$ and $\varphi$ can be produced via
\begin{equation}
\psi_{\theta,\beta_0,\beta_1} = \chi_{I_\theta}(\calE_{\Phi,\theta,\beta_0,\beta_1}) \delta_0, \quad
\varphi_{\theta,\beta_0,\beta_1} = \delta_0 - \psi_{\theta,\beta_0, \beta_1}.
\end{equation}
Since the spectral projection is an orthogonal projection, \eqref{eq:qp:quasiball:phipsiconstr1} is automatically satisfied.

For $n<m$ denote by $\Szego_{[n,m]}(z,\calE_{\Phi, \theta, \beta_0, \beta_1})$ the Szeg\H{o} matrix of $\calE_{\Phi, \theta, \beta_0, \beta_1}$ defined in \eqref{eq:Szegö_prod}.
Note first that applying   \cite[Lemma 2.2]{fillmanSpectralApproximationErgodic2017} gives that $\Szego_{[1,2mq]}(z, \calE_{\Phi, \theta, \beta_0, \beta_1})$ is uniformly bounded for all $\theta \in \bbT$, $z \in I_\theta$, and $m \in \bbZ$. 
Interpolating, one sees that $\Szego_{[1,n]}(z, \calE_{\Phi, \theta, \beta_0, \beta_1})$ is uniformly bounded for $\theta \in \bbT$, $z \in I_\theta$, and $n \in \bbZ$. 
The discussion above implies the uniform bound
\begin{equation} \label{eq:qp:quasiball:unifr2bds}
\sup \{ \|\Szego_{[1,n]}(z, \calE_{\Phi, \theta, \beta_0, \beta_1})\| : \theta \in \bbT, \ z \in I_\theta, \ \beta_j \in \bbD, \ n \in \bbZ\} < \infty.
\end{equation}
This suffices to establish upper and lower bounds on the boundary values of the Carath\'eodory function, which in turn yields the desired estimates \cite[Section~5]{LiDamZho2022TAMS}.
More specifically, the desired upper bound is formulated as \cite[Lemma~5.2]{LiDamZho2022TAMS} but the needed lower bound follows from the same argument by utilizing the corresponding lower bound from the Jitomirskaya--Last inequality.
\end{proof}

We come now to the final lemma, which gives an inductive scheme to construct a suitable $\Phi$ exhibiting the desired almost-ballistic motion uniformly in $\theta$, $\beta_0$, and $\beta_1$.

\begin{lemma}\label{l.almostballisticchoiceofalpha}
There exists an irrational $\Phi \in \bbT$ such that for every $\theta \in \bbT$ and $\beta_0, \beta_1 \in \bbD$, $\calE = \calE_{\Phi, \theta, \beta_0, \beta_1}$ obeys \eqref{e.almostballisticcriterion}.
\end{lemma}

\begin{proof}

We inductively pick $\Phi_m \in \bbQ$, $T_m$, $\Delta_m$ with the following properties:
\begin{itemize}

\item[(i)] $\Phi_{m+1} - \Phi_m = 2^{-k_m!}$ for some $k_m \in \bbN$ satisfying $k_m<k_{m+1}$ for all $m$.

\item[(ii)] For every $\theta \in \bbT$, $\beta_0,\beta_1 \in \bbD$, and $\Phi$ such that $|\Phi - \Phi_m| < \Delta_m$, we have
\begin{equation}\label{e.almostballisticcriterionspecific}
\frac{1}{T_m} \sum_{t=T_m}^{2T_m-1} \sum_{|n| \ge \frac{T_m}{f(T_m)^{1/5}}} |\langle \delta_n, \calE^t_{\Phi, \theta, \beta_0, \beta_1} \delta_0 \rangle|^2 
\ge \frac{1}{f(T_m)^{2/5}}.
\end{equation}

\item[(iii)] $|\Phi_{m+1} - \Phi_k| < \Delta_k$ for every $k = 1,\ldots,m$.

\end{itemize}

We start with $\Phi_1 = 0$ and explain how to pick $T_m$, $\Delta_m$, $\Phi_{m+1}$ given $T_1, \ldots, T_{m-1}$, $\Delta_1, \ldots, \Delta_{m-1}$, $\Phi_1, \ldots, \Phi_{m}$.

Note that the choice of $\Phi_1$ and property (i) ensure both that each $\Phi_m$ is rational and that the sequence $\{ \Phi_m \}$ converges to an irrational limit $\Phi  = \lim \Phi_m$. Indeed, the assumptions above imply that $\Phi_{m+1} = p_m/2^{k_m!}$ with $p_m \in \bbN$ and
\begin{equation} \label{eq:qpQuasiBall:alphamBd}
|\Phi - \Phi_{m+1}| 
= \sum_{j=m+1}^\infty 2^{-k_j!}
\leq 2^{-k_{m+1}!+1}.
\end{equation}

If $\Phi$ were rational, say $\Phi = r/s$, then one would have 
\[
|\Phi - \Phi_{m+1}| = 
\left| \frac{r}{s} - \frac{p_m}{2^{k_m!}} \right| \geq \frac{1}{s \, 2^{k_m!}} \]
for any $m$. Thus, since $2^{k_m!-k_{m+1}!+1} \to 0$, this would produce a contradiction with \eqref{eq:qpQuasiBall:alphamBd}. It follows that $\Phi$ is irrational.

Given $\Phi_m$, apply Lemma~\ref{l.almostballisticppdecomposition} and let $C_1 = C_1(\Phi_m)$ and $C_2 = C_2(\Phi_m)$, using the notation from that lemma. 
Choose $T_m \ge 2 T_{m-1}$ (along with the requirement $T_1 \ge 2$, which then gives $T_m \ge 2^m$) so that
$$
C_1^2 - 3 \sqrt{\frac{2 C_2^2}{f(T_m)^{1/5}}} \ge \frac{2}{f(T_m)^{2/5}}. 
$$
This is possible since $C_1, C_2$ are fixed at this stage and $f(T) \to \infty$ as $T \to \infty$. 
 
Given $\theta \in \bbT$ and $\beta_0, \beta_1\in \bbD$, consider the decomposition
$$
\delta_0 = \varphi_{\theta,\beta_0, \beta_1} + \psi_{\theta,\beta_0, \beta_1}
$$
provided by Lemma~\ref{l.almostballisticppdecomposition}.

Denoting the orthogonal projection onto $\spn\{ \delta_n : |n| < \frac{T_m}{f(T_m)^{1/5}} \}$ by $P_m$, we first note using Lemma~\ref{lem:specFourier} that
\begin{align*}
\frac{1}{T_m} \sum_{t=T_m}^{2T_m-1} \| P_m \calE_{\Phi_m, \theta, \beta_0, \beta_1}^t \psi_{\theta, \beta_0, \beta_1}\|^2  &\le \frac{1}{T_m} \# \left\{ n : |n| < \frac{T_m}{f(T_m)^{1/5}} \right\}  \vertiii{\psi_{\omega,\lambda_0, \lambda_1}}^2 \\
& \le \frac{2 C_2^2}{f(T_m)^{1/5}} .
\end{align*}
Together with Lemma~\ref{l.almostballisticbreakdown}, this yields the lower bound
\begin{align*}
\frac{1}{T_m} 
& \sum_{t=T_m}^{2T_m-1} \sum_{|n| \ge \frac{T_m}{f(T_m)^{1/5}}} |\langle \delta_n, \calE_{\Phi_m, \theta,  \beta_0, \beta_1}^t \delta_0 \rangle|^2  \\
& = \frac{1}{T_m} \sum_{t=T_m}^{2T_m-1} \| (\idty - P_m) \calE_{\Phi_m, \theta,  \beta_0, \beta_1}^t \delta_0\|^2   \\
& \ge \| \psi \|^2 - 3 \left( \frac{1}{T_m} \sum_{t=T_m}^{2T_m-1} \| P_m \calE_{\Phi_m, \theta,  \beta_0, \beta_1}^t \psi_{\Phi_m, \theta, \beta_0, \beta_1}\|^2   \right)^{1/2} \\
& \ge C_1^2 - 3 \sqrt{\frac{2 C_2^2}{f(T_m)^{1/5}}}  \\
& \ge \frac{2}{f(T_m)^{2/5}}.
\end{align*}

By Lemma~\ref{eq:qp:quasiball:dynDuhaml} we can then pick $\Delta_m > 0$ such that for any $\Phi \in \bbT$ with $|\Phi - \Phi_m| \leq \Delta_m$, we have
$$
\frac{1}{T_m} \sum_{t=T_m}^{2T_m-1} \sum_{|n| \ge \frac{T_m}{f(T_m)^{1/5}}} |\langle \delta_n, \calE_{\Phi_m, \theta,  \beta_0, \beta_1}^t \delta_0 \rangle|^2   \ge \frac{1}{f(T_m)^{2/5}}.
$$

Finally, pick $k_m \ge k_{m-1}+1$ (with the convention $k_0=0$) large enough so that $|\Phi_{m+1} - \Phi_k| < \Delta_k$ for $k = 1, \ldots, m$.

This procedure ensures that (i)--(iii) hold for all $m$ and hence the limit $\Phi = \lim_{m \to \infty} \Phi_m$ is irrational and satisfies the following for every $\theta \in \bbT$, $\beta_0, \beta_1 \in \bbD$, and $m$:
\begin{equation}\label{e.almostballisticcriterionspecificconclude}
\frac{1}{T_m} \sum_{t=T_m}^{2T_m-1} \sum_{|n| \ge \frac{T_m}{f(T_m)^{1/5}}} |\langle \delta_n,\calE_{\Phi_m, \theta,  \beta_0, \beta_1}^t \delta_0 \rangle|^2  \ge \frac{1}{f(T_m)^{2/5}}.
\end{equation}
On account of Lemma~\ref{eq:qp:quasiball:dynDuhaml}, we deduce that  \eqref{e.almostballisticcriterion} holds for $\calE_{\Phi, \theta, \beta_0, \beta_1}$, concluding the proof.
\end{proof}

\begin{proof}[Proof of Theorem~\ref{t:pathologicalExample}]
Let a monotone $f : \bbR_+ \to \bbR_+$ such that $\lim_{t \to \infty} f(t) = \infty$ be given. 
The operator $\calE$ will be of the form \eqref{eq:Verblunskys_UAMO_mod}, where $\Phi \in \bbT$ irrational is selected via an application of Lemma~\ref{l.almostballisticchoiceofalpha} and $\beta_j \in \bbD$.

This ensures that the almost ballistic transport statement \eqref{eq:quasiBall} holds for the operator $\calE = \calE_{\Phi, \theta, \beta_0, \beta_1}$ of the form \eqref{eq:quasiBall} with this choice of $\Phi$, uniformly for all choices of $\theta \in \bbT$ and $\beta_0, \beta_1 \in \bbD$. On the other hand, due to a whole-line spectral averaging result in the spirit of \cite[Theorem~10.2.5]{Simon2005OPUC2}\footnote{Strictly speaking, \cite{Simon2005OPUC2} works for varying the leftmost Verblunsky coefficient in the half-line case; however, it is possible to generalize the result to the extended CMV setting. \textit{Inter alia}, this will be addressed in a forthcoming work of the authors \cite{CFVpreprint}.} and \cite[Theorem~2.9]{CFO2023CMP}, there are $\theta \in \bbT$ and $\beta_0, \beta_1 \in \bbD$ so that $\calE_{\Phi, \theta, \beta_0, \beta_1}$ has pure point spectrum with exponentially decaying eigenfunctions (indeed, almost every choice works), concluding the proof.
\end{proof}

\begin{proof}[Proof of Theorem~\ref{t:quasiballistic}]
    This is an immediate consequence of Theorem~\ref{t:pathologicalExample}.
\end{proof}

\bibliographystyle{abbrvArXiv}
\bibliography{refs}

\end{document}